\documentclass[10pt]{article}
\usepackage{amsmath,amssymb,amsthm}

\theoremstyle{plain}
\newtheorem{theorem}{Theorem}

\newtheorem{corollary}[theorem]{Corollary}

\newtheorem*{qqq}{Theorem~\ref{HTplus}}
\newtheorem*{ppp}{Theorem~\ref{FG}}

\theoremstyle{definition}

\title{\bf Several combinatorial results generalized from one large subset of semigroups to infinitely many}
\date{}
\author{Teng Zhang
        \footnote{School of Science, Zhejiang University of Science and
Technology, Liuhe Road, Hangzhou, 310023, Zhejiang, China\hfill\break
{\tt zhangteng@zust.edu.cn}\hfill\break
{Keywords: finite sums, Hales-Jewett Theorem, the Central Set Theorem, Ramsey Theory \hfill\break
}}
}

\begin{document}
\maketitle

\begin{abstract}
  In 2015, Phulara established a generalization of the famous central set theorem by an original idea. Roughly speaking, this idea extends a combinatorial result from one large subset of the given semigroup to countably many. In this paper, we apply this idea to other combinatorial results to obtain corresponding generalizations, and do some further investigation. Moreover, we find that Phulara's generalization can be generalized further that can deal with uncountably many C-sets.
\end{abstract}

\section{Introduction}
In 1981, Furstenberg established the famous central set theorem\cite[Proposition 8.21]{1981Recurrence}, which has many applications. For example, Furstenberg proved \cite[Theorem 8.22]{1981Recurrence} that any finite system of equations satisfying Rado's columns condition has solutions in any central set by this theorem. After that, many mathematicians devoted to extending this theorem from different aspects, such as \cite{1990Nonmetrizable, 1992Ramsey, 2021Central, 2015A}. Where a significant strengthening was obtained in \cite{2008A}, which can deal with all sequences at once, rather than countably many. But all these generalizations only consider one central set. In 2015, Phulara\cite{2015Ag} obtained a generalization which consider countably many C-sets in the same idempotent in $J(S)$ (it is also correct if C-sets are replaced by central sets), and he used this generalization to obtain some applications. 

After observation, we found that the idea of Phulara's generalization may be used to extend other combinatorial results. So in Section 2, we will apply this idea to several known results, in particular the following three results:
\begin{enumerate}
  \item[(i)] the property of central sets with respect to arithmetic progressions (i.e. every central set in $(\mathbb{N}, +)$ contains arbitrarily long arithmetic progressions);
  \item[(ii)] the property of IP sets with respect to finite sums (\cite[Theorem 5.8]{2012Algebra}); and
  \item[(iii)] the property of central sets in free semigroups (Theorem \ref{HJ}).
\end{enumerate}
Notice that the first result implies van der Waerden's Theorem (\cite[Corollary 14.2]{2012Algebra}). The second one implies Hindman's Theorem \cite[Corollary 5.10]{2012Algebra}) and the third one implies the Hales-Jewett Theorem (\cite[Corollary 14.8]{2012Algebra}). And we obtain the corresponding generalizations: Theorem \ref{vdw}, Theorem \ref{HT} and Theorem \ref{HJplus}.

Moreover, we consider whether these generalizations can be extended further to make the conclusion deal with finite sums of infinite sets (or sequences) rather than finite sums of finite sets (or sequences). We find that the first one and the third one fail, but the second one can go through as follows:
\begin{qqq}
  Suppose $(S, +)$ is a commutative semigroup, $p$ is an idempotent in $(\beta S, +)$ and $\langle C_n \rangle_{n=1}^\infty$ is a sequence in $p$. Then there exists for each $i \in \mathbb{N}$ a sequence $\langle x_{i, j}\rangle_{j=1}^\infty$ in $S$ satisfying $FS(\langle x_{i, j}\rangle_{j=1}^\infty) \subseteq C_i$, such that for each $F \in \mathcal{P}_f(\mathbb{N})$, $\sum_{i \in F} \langle x_{i, j}\rangle_{j=1}^\infty \subseteq C_{\min F}$.
\end{qqq}

In the last section, we extend Phulara's result further and establish the following result in commutative semigroups:
\begin{ppp}
  Suppose $(S, +)$ is a commutative semigroup, $r \in J(S)$ is an idempotent and $R: \mathcal{P}_f({^\mathbb{N}}S) \rightarrow r$ is a function. Then there exist functions $\alpha: \mathcal{P}_f({^\mathbb{N}}S) \rightarrow S$ and $H: \mathcal{P}_f({^\mathbb{N}}S) \rightarrow \mathcal{P}_f(\mathbb{N})$ such that
  \begin{enumerate}
    \item if $F, G \in \mathcal{P}_f({^\mathbb{N}}S)$ and $F \subsetneq G$, then $\max H(F) < \min H(G)$, and
    \item if $m \in \mathbb{N}$, $G_1, \ldots, G_m \in \mathcal{P}_f({^\mathbb{N}}S)$, $G_1 \subsetneq \ldots \subsetneq G_m$ and $f_i \in G_i$ for each $i \in \{1, \ldots, m\}$, then $\sum_{i=1}^{m}(\alpha(G_i) + \sum_{t \in H(G_i)}f_i(t)) \in R(G_1)$.
  \end{enumerate}
\end{ppp}
Comparing with Phulara's result\cite[Theorem 2.6]{2015Ag}, our result deal with at most $\kappa^\omega$ many C-sets at once ($\kappa$ is the size of the semigroup), rather than only countably many. And this generalization also goes through in noncommutative semigroups(Theorem \ref{noncom}).

Now let us introduce some notions, notations and basic facts that we will refer to, most of these information can be found in \cite{2012Algebra}. Given a discrete semigroup $(S, \cdot)$, $\beta S$ is the Stone-\v{C}ech compactification of $S$ (it is the set of all ultrafilters on $S$) and there is a natural extension of $\cdot$ to $\beta S$ making $\beta S$ a compact right topological semigroup. So for $p, q \in \beta S$, $p \cdot q = \{A \subseteq S: \{x \in S: x^{-1}A \in q \} \in p \}$, where $x^{-1}A = \{y \in S: xy \in A \}$. And for each $p \in \beta S$, the function $\rho_p: \beta S \rightarrow \beta S$, defined by $\rho_p(q) = q \cdot p$, is continuous, and for each $x \in S$, $\lambda_x: \beta S \rightarrow \beta S$, defined by $\lambda_x(p) = x \cdot p$, is also continuous.The topological basis of $\beta S$ is $\{U_A: \emptyset \neq A \subseteq S \}$, where $U_A = \{p \in \beta S: A \in p \}$. Given a compact right topological semigroup $S$, it has a smallest nonempty ideal $K(S)$, which is the union of all minimal left ideals of $S$ and also the union of all minimal right ideals of $S$, an idempotent in $K(S)$ is called minimal.

Let $(S, +)$ be a semigroup, $\langle x_n \rangle_{n=1}^\infty$ be a sequence in $S$, write $\mathrm{FS}(\langle x_n \rangle_{n=1}^\infty) = \{\sum_{n \in H}x_n: H \in \mathcal{P}_f(\mathbb{N}) \}$, where $\mathcal{P}_f(\mathbb{N})$ is the set of nonempty finite subsets of $\mathbb{N}$ and $\sum_{n \in H}x_n$ is the sum in increasing order of indices. Similarly, if $k \in \mathbb{N}$ and $\langle x_n \rangle_{n=1}^k$ be a sequence in $S$, we denote $\mathrm{FS}(\langle x_n \rangle_{n=1}^k) = \{\sum_{n \in H}x_n: H \in \mathcal{P}_f(\mathbb{N})$ and $\max H \leq k \}$. A subset $A$ of $S$ is called an IP set if there exists a sequence $\langle x_n \rangle_{n=1}^\infty$ in $S$ such that $\mathrm{FS}(\langle x_n \rangle_{n=1}^\infty) \subseteq A$; $A$ is called a central set if there exists a minimal idempotent $p \in \beta S$ such that $A \in p$. See \cite[Definition 14.14.1]{2012Algebra} for the definition of J-sets, C-sets and $J(S)$; if the semigroup is commutative, then the first two of these are simpler (\cite[Definition 14.8.1, Definition 14.8.5]{2012Algebra}).

Here we emphasize a notation\cite[Definition 4.13]{2012Algebra} which will be applied frequently. If $(S, \cdot)$ is a semigroup, $A \subseteq S$ and $p \in \beta S$, then denote $A^\star(p) = \{ s \in A: s^{-1}A \in p \}$. For convenience, we write $A^\star$ instead of $A^\star(p)$ without causing ambiguity. Moreover, If $p$ is an idempotent and $A \in p$, then by \cite[Lemma 4.14]{2012Algebra}, $s^{-1}A^\star \in p$ for each $s \in A^\star$.
\section{Several Applications}
In this section, we shall apply Phulara's idea to several known result to establish the corresponding generalizations.

For convenience, if $F$ is a finite subset of $\mathbb{N}$ and $Y_n$ is a nonempty set (or sequence) for each $n \in F$, then we denote $\sum_{n \in F} Y_n = \{ \sum_{n \in F}a_n: a_n$ is a point in $Y_n$ for each $n \in F \}$. Then we have the following result, which is a generalization of \cite[Theorem 16.16]{2012Algebra}. See \cite[Definition 15.1]{2012Algebra} for image partition regular matrices.
\begin{theorem}\label{matrix}
  Suppose $p$ is a minimal idempotent in $(\beta\mathbb{N}, +)$, $\langle B_n \rangle_{n=1}^\infty$ is a sequence in $p$, $\langle A_n \rangle_{n=1}^\infty$ is a sequence of finite image partition regular matrices with entries from $\mathbb{Q}$ and $m_n$ is the number of columns of $A_n$ for each $n \in \mathbb{N}$. Then for each $n \in \mathbb{N}$, there exists $\overrightarrow{x_n} \in \mathbb{N}^{m_n}$ such that for each $F \in \mathcal{P}_f(\mathbb{N})$, $\sum_{n \in F} Y_n \subseteq B_{\min F}$, where $Y_n$ is the set of entries of $A_n \overrightarrow{x_n}$ for each $n \in \mathbb{N}$.
\end{theorem}
\begin{proof}
  Note that $B_1^\star$ is central, so by \cite[Theorem 16.14(a)]{2012Algebra}, pick some $\overrightarrow{x_1} \in \mathbb{N}^{m_1}$ such that all entries of $A_1 \overrightarrow{x_1}$ are in  $B_1^\star$. Let $Y_1$ be the set of entries of $A_1 \overrightarrow{x_1}$. 
  
  Inductively, assume $n \in \mathbb{N}$ and we have chosen $\overrightarrow{x_k} \in \mathbb{N}^{m_k}$ and $Y_k \subseteq \mathbb{N}$ for each $k \in \{ 1, \ldots, n \}$ such that $Y_k$ is the set of entries of $A_k \overrightarrow{x_k}$ and for each $F \subseteq \{ 1, \ldots, n \}$, $\sum_{k \in F} Y_k \subseteq B_{\min F}^\star$. Then for each $k \in \{ 1, \ldots, n \}$, define $M_k = \bigcup\{\sum_{t \in F} Y_t: \emptyset \neq F \subseteq \{ 1, \ldots, n \}$ and $\min F = k \}$ and $B = B_{n+1}^\star \cap \bigcap_{k=1}^n\bigcap_{a \in M_k}(-a + B_k^\star)$. Note that each $M_k$ is finite and $M_k \subseteq B_k^\star$ by inductive hypothesis, it turns out that $B \in p$. Then we apply \cite[Theorem 16.14(a)]{2012Algebra} again to obtain $\overrightarrow{x_{n+1}} \in \mathbb{N}^{m_{n+1}}$ such that all entries of $A_{n+1} \overrightarrow{x_{n+1}}$ are in  $B$. Then let $Y_{n+1}$ be the set of entries of $A_{n+1} \overrightarrow{x_{n+1}}$.
  
  Next let us verify that for each $F \subseteq \{ 1, \ldots, n+1 \}$, $\sum_{k \in F} Y_k \subseteq B_{\min F}^\star$. Pick $F \subseteq \{ 1, \ldots, n+1 \}$ arbitrarily. If $n + 1 \notin F$, then it hold by hypothesis; otherwise, $n+1 \in F$. If $\{ n+1 \} = F$, then $\sum_{k \in F} Y_k = Y_{n+1} \subseteq B \subseteq B_{n+1}^\star = B_{\min F}^\star$; otherwise $|F| > 1$. Let $G \in F \setminus \{ n+1 \}$, so $\min F = \min G$ and $\max G \leq n$. Observe that for each $x \in \sum_{k \in F} Y_k$, there is some $a \in \sum_{k \in G} Y_k$ and $b \in Y_{n+1}$ such that $x = a + b$ and $a \in M_{\min G}$. Hence $b \in Y_{n+1} \subseteq B \subseteq -a + B_{\min G}^\star = -a + B_{\min F}^\star$. It turns out that $x = a + b \in B_{\min F}^\star$ and so $\sum_{k \in F} Y_k \subseteq B_{\min F}^\star$.
\end{proof}

By the similar argument, we have the following result which is a strengthening of \cite[Theorem 16.17]{2012Algebra}. See \cite[Definition 15.12]{2012Algebra} for kernel partition regular matrices.
\begin{theorem}
  Suppose $p$ is a minimal idempotent in $(\beta\mathbb{N}, +)$, $\langle B_n \rangle_{n=1}^\infty$ is a sequence in $p$, $\langle A_n \rangle_{n=1}^\infty$ is a sequence of kernel partition regular matrices with entries from $\mathbb{Q}$ and $m_n$ is the number of columns of $A_n$ for each $n \in \mathbb{N}$. Then for each $n \in \mathbb{N}$, there exists $\overrightarrow{x_n} \in \mathbb{N}^{m_n}$ such that $A_n \overrightarrow{x_n} = \overrightarrow{0}$ and for each $F \in \mathcal{P}_f(\mathbb{N})$, $\sum_{n \in F} Y_n \subseteq B_{\min F}$, where $Y_n$ is the set of entries of $\overrightarrow{x_n}$ for each $n \in \mathbb{N}$.
\end{theorem}
\begin{proof}
  The proof is nearly the same as that of Theorem \ref{matrix} except that we need to apply \cite[Theorem 16.14(b)]{2012Algebra} instead of \cite[Theorem 16.14(a)]{2012Algebra}.
\end{proof}

The idea can also be applied to the property of central sets with respect to arithmetic progressions (i.e. every central set in ($\mathbb{N}, +)$ contains arbitrarily long arithmetic progressions), which deduces van der Waerden's Theorem (\cite[Corollary 14.2]{2012Algebra}).
\begin{theorem}\label{vdw}
  Suppose $p$ is a minimal idempotent in $(\beta\mathbb{N}, +)$, $\langle C_n \rangle_{n=1}^\infty$ is a sequence in $p$ and $\langle a_n \rangle_{n=1}^\infty$ is a sequence in $\mathbb{N}$. Then there exists for each $n \in \mathbb{N}$ an arithmetic progression $Y_n$ of length $a_n$ in $C_n$ such that for each $F \in \mathcal{P}_f(\mathbb{N})$, $\sum_{n \in F} Y_n \subseteq C_{\min F}$.
\end{theorem}
\begin{proof}
  Note that $C_1^\star$ is a central set, pick an arithmetic progression $Y_1$ of length $a_1$ in $C_1^\star$. Assume $n \in \mathbb{N}$ and we have chosen $Y_k$ for each $k \in \{ 1, \ldots, n \}$ such that $Y_k$ is an arithmetic progression of length $a_k$ in $C_k^\star$ and for each $F \in \{ 1, \ldots, n \}$, $\sum_{k \in F} Y_k \subseteq C_{\min F}^\star$. For each $k \in \{ 1, \ldots, n \}$, define $M_k = \bigcup\{\sum_{t \in F} Y_t: \emptyset \neq F \subseteq \{ 1, \ldots, n \}$ and $\min F = k \}$ and $C = C_{n+1}^\star \cap \bigcap_{k=1}^n\bigcap_{a \in M_k}(-a + C_k^\star)$. Observe that each $M_k$ is finite and $M_k \subseteq C_k^\star$ by inductive hypothesis, so $C \in p$, which means $C$ is central. Then pick an arithmetic progression $Y_{n+1}$ of length $a_{n+1}$ in $C$, which is as desired. 
\end{proof}

One may consider whether this idea apply to the property of IP sets with respect to finite sums (\cite[Theorem 5.8]{2012Algebra}), which deduces Hindman's Theorem (\cite[Corollary 5.10]{2012Algebra}). After observation of above three generalizations, we notice that the conclusion only deals with finite sums of finite sets or sequences. Hence if we want to obtain a generalization of \cite[Theorem 5.8]{2012Algebra} in $(\mathbb{N}, +)$ by this idea, the following result is optimal:
\begin{theorem}\label{HT}
  Suppose $p$ is an idempotent in $(\beta\mathbb{N}, +)$, $\langle C_n \rangle_{n=1}^\infty$ is a sequence in $p$ and $\langle a_n \rangle_{n=1}^\infty$ is a sequence in $\mathbb{N}$. Then there exists for each $n \in \mathbb{N}$ a sequence $Y_n$ in $\mathbb{N}$ with length $a_n$ satisfying $FS(Y_n) \subseteq C_n$, such that for each $F \in \mathcal{P}_f(\mathbb{N})$, $\sum_{n \in F} Y_n \subseteq C_{\min F}$.
\end{theorem}
\begin{proof}
  The proof is essential the same as that of Theorem \ref{vdw} except that we need to apply the property of IP sets that each IP set in $(\mathbb{N}, +)$ contains an arbitrarily long sequence and its finite sum.
\end{proof}
When all $C_n$'s are the same and all $a_n$'s are picked by 1, Theorem \ref{HT} is exactly \cite[Theorem 5.8]{2012Algebra} in $(\mathbb{N}, +)$. However, one may notice that the proof of Theorem \ref{HT} does not use the full property of IP sets with respect to finite sums, it only needs that each $C_n$ is IP$_0$ (\cite[p5]{2020On}), although we need all $C_n$'s to be in the same idempotent that still implies each $C_n$ is IP. So it is natural to ask whether Theorem \ref{HT} still holds if one requests each $Y_n$ to be infinite. Here we give a positive answer in commutative semigroups:
\begin{theorem}\label{HTplus}
  Suppose $(S, +)$ is a commutative semigroup, $p$ is an idempotent in $(\beta S, +)$ and $\langle C_n \rangle_{n=1}^\infty$ is a sequence in $p$. Then there exists for each $i \in \mathbb{N}$ a sequence $\langle x_{i, j}\rangle_{j=1}^\infty$ in $S$ satisfying $FS(\langle x_{i, j}\rangle_{j=1}^\infty) \subseteq C_i$, such that for each $F \in \mathcal{P}_f(\mathbb{N})$, $\sum_{i \in F} \langle x_{i, j}\rangle_{j=1}^\infty \subseteq C_{\min F}$.
\end{theorem}
\begin{proof}
  Without loss of generality, we assume that $C_n \supseteq C_m$ whenever $n \leq m$ (if not for some $C_n$, let $\bigcap_{1 \leq i \leq n}C_n$ be the new $C_n$). We build $\langle \langle x_{i, j} \rangle_{j=1}^\infty \rangle_{i=1}^\infty$ by induction. First pick $x_{1, 1} \in C_1^\star$ arbitrarily.
  
  Assume $n \in \mathbb{N}$ and we have chosen $x_{i, j}$ for each $i, j \in \{1, \ldots, n \}$ such that $\mathrm{FS}(\langle x_{i, j}\rangle_{j=1}^n) \subseteq C_i^\star$ for each $i \in \{ 1, \ldots, n \}$ and $\sum_{i \in F} \langle x_{i, j}\rangle_{j=1}^n \subseteq C_{\min F}^\star$ for each nonempty $F \subseteq \{ 1, \ldots, n \}$. If $A \in p$ and $X \in \mathcal{P}_f(A^\star)$, we denote $B_X^A = A^\star \cap \bigcap_{x \in X}(-x + A^\star)$. Hence $B_X^A \in p$. For $m \in \{ 1, \ldots, n \}$, define a function $t_m: \{ 1, \ldots, n \} \setminus \{ m \} \rightarrow \{ n, n+1 \}$ by setting $t_m(i) = n+1$ if $i < m$ and $t_m(i) = n$ otherwise. Now let us build $\langle x_{i, n+1} \rangle_{i=1}^n$. 
  
  If $n = 1$, then $\mathrm{FS}(\langle x_{1, j}\rangle_{j=1}^n) = \{ x_{1, 1} \} \subseteq C_1^\star$, so $B_{\mathrm{FS}(\langle x_{1, j}\rangle_{j=1}^n)}^{C_1} \in p$ and we pick a point $x_{1, n+1} \in B_{\mathrm{FS}(\langle x_{1, j}\rangle_{j=1}^n)}^{C_1}$. Otherwise $n > 1$. By inductive hypothesis $\mathrm{FS}(\langle x_{1, j}\rangle_{j=1}^n) \subseteq C_1^\star$ and $\sum_{i \in F} \langle x_{i, j}\rangle_{j=1}^n \subseteq C_{\min F}^\star$ for each nonempty $F \subseteq \{ 1, \ldots, n \}$, we have $B_{\mathrm{FS}(\langle x_{1, j}\rangle_{j=1}^n)}^{C_1} \in p$ and $\bigcap_{\emptyset \neq F \subseteq \{ 2, \ldots, n \}}B_{\sum_{i \in F}\langle x_{i, j} \rangle_{j=1}^n}^{C_{\min F}} \in p$ so pick 
  \[ x_{1, n+1} \in B_{\mathrm{FS}(\langle x_{1, j}\rangle_{j=1}^n)}^{C_1} \cap \bigcap_{\emptyset \neq F \subseteq \{ 2, \ldots, n \}}B_{\sum_{i \in F}\langle x_{i, j} \rangle_{j=1}^n}^{C_{\min F}}. \]
   Generally, if $m \in \{1, \ldots, n-1 \}$ and for each $k \in \{ 1, \ldots, m \}$, we have picked $x_{k, n+1} \in B_{\mathrm{FS}(\langle x_{k, j}\rangle_{j=1}^n)}^{C_k} \cap \bigcap_{\emptyset \neq F \subseteq \{ 1, \ldots, n \} \setminus \{ k \}}B_{\sum_{i \in F}\langle x_{i, j} \rangle_{j=1}^{t_k(i)}}^{C_{\min F}}$, let us show that $ B_{\mathrm{FS}(\langle x_{m+1, j}\rangle_{j=1}^n)}^{C_{m+1}} \cap \bigcap_{\emptyset \neq F \subseteq \{ 1, \ldots, n \} \setminus \{ m+1 \}}B_{\sum_{i \in F}\langle x_{i, j} \rangle_{j=1}^{t_{m+1}(i)}}^{C_{\min F}} \in p$ so that we can pick $x_{m+1, n+1}$ from it. By hypothesis we immediately obtain $B_{\mathrm{FS}(\langle x_{m+1, j}\rangle_{j=1}^n)}^{C_{m+1}} \in p$.
   
   Now take nonempty $F \subseteq \{ 1, \ldots, n \} \setminus \{ m+1 \}$ and $x \in \sum_{i \in F}\langle x_{i, j} \rangle_{j=1}^{t_{m+1}(i)}$ arbitrarily, we shall show $x \in C_{\min{F}}^\star$. If $x \in \sum_{i \in F}\langle x_{i, j} \rangle_{j=1}^n$, by hypothesis $x \in C_{\min{F}}^\star$; otherwise, either $|F| = 1$, then there is some $k < m+1$ such that $x = x_{k, n+1}$, or $|F| > 1$, then there is some $k < m+1$ such that $x = y + x_{k, n+1}$ for some $y \in \sum_{i \in F \setminus \{ k \}}\langle x_{i, j} \rangle_{j=1}^{t_k(i)}$. If the former holds, $x = x_{k, n+1} \in B_{\mathrm{FS}(\langle x_{k, j}\rangle_{j=1}^n)}^{C_k} \subseteq C_k^\star = C_{\min{F}}^\star$. If the latter holds, then $x_{k, n+1} \in B_{\sum_{i \in F \setminus \{ k \}}\langle x_{i, j} \rangle_{j=1}^{t_k(i)}}^{C_{\min{F \setminus \{ k \}}}} \subseteq -y + C_{\min{F \setminus \{ k \}}}^\star$, so $x \in C_{\min{F \setminus \{ k \}}}^\star \subseteq C_{\min{F}}^\star$. Therefore, $\sum_{i \in F}\langle x_{i, j} \rangle_{j=1}^{t_{m+1}(i)} \subseteq C_{\min{F}}^\star$, which guarantees that $B_{\sum_{i \in F}\langle x_{i, j} \rangle_{j=1}^{t_{m+1}(i)}}^{C_{\min F}} \in p$ for any nonempty $F \subseteq \{ 1, \ldots, n \} \setminus \{ m+1 \}$. Then we pick a point $x_{m+1, n+1} \in B_{\mathrm{FS}(\langle x_{m+1, j}\rangle_{j=1}^n)}^{C_{m+1}} \cap \bigcap_{\emptyset \neq F \subseteq \{ 1, \ldots, n \} \setminus \{ m+1 \}}B_{\sum_{i \in F}\langle x_{i, j} \rangle_{j=1}^{t_{m+1}(i)}}^{C_{\min F}}$.
  
  By induction, we obtain $\langle x_{i, n+1} \rangle_{i=1}^n$ such that for each $k \in \{ 1, \ldots, n \}$, \[ 
  x_{k, n+1} \in B_{\mathrm{FS}(\langle x_{k, j}\rangle_{j=1}^n)}^{C_k} \cap \bigcap_{\emptyset \neq F \subseteq \{ 1, \ldots, n \} \setminus \{ k \}}B_{\sum_{i \in F}\langle x_{i, j} \rangle_{j=1}^{t_k(i)}}^{C_{\min F}}.\]
   Now let us build $\langle x_{n+1, j} \rangle_{j=1}^{n+1}$ by induction. For convenience, we denote $B = \bigcap_{\emptyset \neq F \subseteq \{ 1, \ldots, n \}}B_{\sum_{i \in F}\langle x_{i, j} \rangle_{j=1}^{n+1}}^{C_{\min F}}$. Notice that for any nonempty $F \subseteq \{ 1, \ldots, n \}$, $\sum_{i \in F}\langle x_{i, j} \rangle_{j=1}^{n+1} \subseteq C_{\min{F}}^\star$, the proof of which is nearly the same as last paragraph, so we omit it in order to avoid a lot of repetition. Hence $B \in p$ so pick $x_{n+1, 1} \in C_{n+1}^\star \cap B$. Observe that $\mathrm{FS}(\langle x_{n+1, j}\rangle_{j=1}^1) = \{ x_{n+1, 1} \} \subseteq C_{n+1}^\star$, so pick $x_{n+1, 2} \in B_{\mathrm{FS}(\langle x_{n+1, j}\rangle_{j=1}^1)}^{C_{n+1}} \cap B$.
  
  Assume $r \in \{ 2, \ldots, n \}$ and we have obtained $\langle x_{n+1, j} \rangle_{j=1}^r$ such that for each $m \in \{ 2, \ldots, r \}$, $\mathrm{FS}(\langle x_{n+1, j}\rangle_{j=1}^{m-1}) \subseteq C_{n+1}^\star$ and $x_{n+1, m} \in B_{\mathrm{FS}(\langle x_{n+1, j}\rangle_{j=1}^{m-1})}^{C_{n+1}} \cap B$. Let us show that $\mathrm{FS}(\langle x_{n+1, j}\rangle_{j=1}^r) \subseteq C_{n+1}^\star$. Take $x \in \mathrm{FS}(\langle x_{n+1, j}\rangle_{j=1}^r)$ arbitrarily. If $x \in \mathrm{FS}(\langle x_{n+1, j}\rangle_{j=1}^{r-1})$, then by inductive hypothesis $x \in C_{n+1}^\star$; otherwise either $x= x_{n+1, r}$ or $x = y + x_{n+1, r}$ for some $y \in \mathrm{FS}(\langle x_{n+1, j}\rangle_{j=1}^{r-1})$. If the former holds, by hypothesis $x \in C_{n+1}^\star$; if the latter holds, then $x_{n+1, r} \in B_{\mathrm{FS}(\langle x_{n+1, j}\rangle_{j=1}^{r-1})}^{C_{n+1}} \cap B \subseteq -y + C_{n+1}^\star$ so $x \in C_{n+1}^\star$. 
  
  Therefore $\mathrm{FS}(\langle x_{n+1, j}\rangle_{j=1}^r) \subseteq C_{n+1}^\star$, so we pick $x_{n+1, r+1} \in B_{\mathrm{FS}(\langle x_{n+1, j}\rangle_{j=1}^r)}^{C_{n+1}} \cap B$. By induction, we obtain $\langle x_{n+1, j} \rangle_{j=1}^{n+1}$ such that $x_{n+1, 1} \in C_{n+1}^\star \cap B$ and for $k \in \{ 2, \ldots, n+1 \}$, \[
  x_{n+1, k} \in B_{\mathrm{FS}(\langle x_{n+1, j}\rangle_{j=1}^{k-1})}^{C_{n+1}} \cap B. \]
  
 Now let us verify that for each $i \in \{ 1, \ldots, n+1 \}$, $\mathrm{FS}(\langle x_{i, j}\rangle_{j=1}^{n+1}) \subseteq C_i^\star$.
  
 1$^\circ$. If $i \in \{ 1, \ldots, n \}$, then for any $x \in \mathrm{FS}(\langle x_{i, j}\rangle_{j=1}^{n+1})$, if $x \in \mathrm{FS}(\langle x_{i, j}\rangle_{j=1}^n)$, then $x \in C_i^\star$ by inductive hypothesis; if $x = x_{i, n+1}$, then by construction $x \in B_{\mathrm{FS}(\langle x_{i, j}\rangle_{j=1}^n)}^{C_i} \subseteq C_i^\star$; otherwise, there is some $y \in \mathrm{FS}(\langle x_{i, j}\rangle_{j=1}^n)$ such that $x = y + x_{i, n+1}$, so $x_{i, n+1} \in B_{\mathrm{FS}(\langle x_{i, j}\rangle_{j=1}^n)}^{C_i} \subseteq -y + C_i^\star$, which deduces $x \in C_i^\star$.
  
 2$^\circ$. Otherwise, $i = n+1$. For any $x \in \mathrm{FS}(\langle x_{i, j}\rangle_{j=1}^{n+1})$, if $x = x_{n+1, 1}$, immediately we have $x \in C_{n+1}^\star$ by construction; if $x = x_{n+1, r}$ for some $r \in \{ 2, \ldots, n+1 \}$, then $x \in B_{\mathrm{FS}(\langle x_{n+1, j}\rangle_{j=1}^{r-1})}^{C_{n+1}} \subseteq C_{n+1}^\star$; otherwise, there exist $r \in \{ 2, \ldots, n+1 \}$ and $y \in \mathrm{FS}(\langle x_{n+1, j}\rangle_{j=1}^{r-1})$ such that $x = y + x_{n+1, r}$, so $x_{n+1, r} \in B_{\mathrm{FS}(\langle x_{n+1, j}\rangle_{j=1}^{r-1})}^{C_{n+1}} \subseteq -y + C_{n+1}^\star$, which deduces $x \in C_{n+1}^\star$.

 Next we verify that for each nonempty $F \subseteq \{ 1, \ldots, n+1 \}$, $\sum_{i \in F}\langle x_{i, j} \rangle_{j=1}^{n+1} \subseteq C_{\min F}^\star$. We have already known that $\sum_{i \in F}\langle x_{i, j} \rangle_{j=1}^{n+1} \subseteq C_{\min{F}}^\star$ for any nonempty $F \subseteq \{ 1, \ldots, n \}$. So we assume that $n + 1 \in F$.
  
 3$^\circ$. If $F = \{ n+1 \}$, then $\sum_{i \in F}\langle x_{i, j} \rangle_{j=1}^{n+1} = \{ x_{n+1, 1}, \ldots , x_{n+1, n+1} \} \subseteq C_{n+1}^\star = C_{\min F}^\star$.
 
 4$^\circ$. Otherwise $|F| > 1$ and $n + 1 \in F$. Then for any $x \in \sum_{i \in F}\langle x_{i, j} \rangle_{j=1}^{n+1}$, there is some $y \in \sum_{i \in F \setminus \{ n+1 \}}\langle x_{i, j} \rangle_{j=1}^{n+1}$ and some $k \in \{ 1, \ldots, n+1 \}$ such that $x = y + x_{n+1, k}$. Note that $x_{n+1, k} \in B \subseteq -y + C_{\min F \setminus \{ n+1 \}}^\star$, so $x \in C_{\min F \setminus \{ n+1 \}}^\star \subseteq C_{\min F}^\star$. 
 
 By induction, we obtain $\langle \langle x_{i, j} \rangle_{j=1}^\infty \rangle_{i=1}^\infty$, which is as desired.
\end{proof}

However, we do not know whether Theorem \ref{HTplus} still holds for noncommutative semigroups, so we leave it as an open question. But one may ask whether Theorem \ref{vdw} has a similar generalization like Theorem \ref{HT}, that is to say, can Theorem \ref{vdw} still hold if each arithmetic progression $Y_n$ is required to be infinite? Unfortunately, it fails badly because central sets may not contain arithmetic progressions of infinite length. 
\begin{theorem}\label{infAP}
  There is no ultrafilter $p \in \beta\mathbb{N}$ such that for each $A \in p$, $A$ contains an arithmetic progression of infinite length.
\end{theorem}
\begin{proof}
  For every $(a, b) \in \mathbb{N} \times \mathbb{N}$, let $\mathrm{AP}_{a,b} = \{a + nb: n \in \mathbb{N} \cup \{0\} \}$ which represents an arithmetic progression of infinite length with $a$ as the first item and $b$ as the tolerance. Then let $\mathcal{A} = \{\mathrm{AP}_{a,b}: (a, b) \in \mathbb{N} \times \mathbb{N} \}$. Note that the size of $\mathcal{A}$ is $\omega$, so that we enumerate $\mathcal{A} = \{l_1, l_2, \ldots, l_n, \ldots \}$. Next we construct two sequences $\langle x_n \rangle_{n=1}^\infty, \langle y_n \rangle_{n=1}^\infty$, both of which will meet every arithmetic progression of infinite length, but which have no intersection with each other. First pick two distinct points $x_1, y_1 \in l_1$. Assume $k \in \mathbb{N}$ and we have $\langle x_i \rangle_{i=1}^k, \langle y_i \rangle_{i=1}^k$ such that for each $n \leq k$, $x_n \notin \{x_i: i\in \{1, \ldots ,n\}\setminus \{n\}\} \cup \{y_i: i\in \{1, \ldots ,n\} \}$ and $y_n \notin \{x_i: i\in \{1, \ldots ,n\}\} \cup \{y_i: i\in \{1, \ldots ,n\}\setminus \{n\} \}$. Then we pick $x_{k+1} \in l_{k+1} \setminus \{x_1, \ldots, x_k, y_1, \ldots, y_k \}$. Then pick $y_{k+1} \in l_{k+1} \setminus \{x_1, \ldots, x_k, x_{k+1}, y_1, \ldots, y_k \}$.

  Finally, we have $A = \{x_n: n \in\mathbb{N} \}$, $B^\prime = \{y_n: n \in\mathbb{N} \}$. Let $B = \mathbb{N} \setminus A$. Note that $A \cap B^\prime = \emptyset$, so $B^\prime \subseteq B$. Now assume that there exists an ultrafilter $p \in \beta\mathbb{N}$ such that for each $A \in p$, $A$ contains an arithmetic progression of infinite length, then either $A$ or $B$ is in $p$. If $A \in p$, since $\mathcal{A}$ is the set of all arithmetic progressions of infinite length, there is some $l_n \subseteq A$, $n \in \mathbb{N}$. While by construction, $y_n \in l_n$ and $y_n \in B$, contradiction. Similarly, if $B \in p$, then there is some $l_m \subseteq B$, $m \in \mathbb{N}$, but $x_n \in l_n$ and $x_n \in A$, this is also a contradiction.
\end{proof}

By the above result, we immediately obtain the following corollary:
\begin{corollary}
  There exists a central set in $(\mathbb{N}, +)$ which contains no arithmetic progression of infinite length. 
\end{corollary}
\begin{proof}
  Let $p$ be a minimal idempotent in $(\beta\mathbb{N}, +)$. By Theorem \ref{infAP}, there exists some $A \in p$ containing no arithmetic progression of infinite length, meanwhile $A$ is central in $(\mathbb{N}, +)$.
\end{proof}

This idea can also be applied to the the following result, which can immediately deduce the Hales-Jewett Theorem (\cite[Corollary 14.8]{2012Algebra}). See \cite[Definition 1.3, Definition 14.6]{2012Algebra} for the definitions of free semigroups and variable words.
\begin{theorem}\label{HJ}
   Suppose $A$ is a nonempty finite alphabet, $S$ is the free semigroup over $A$ and $C$ is a central set in $S$. Then there is a variable word $w(v)$ such that $\{ w(a): a \in A \} \subseteq C$.
\end{theorem}
\begin{proof}
  It is a direct corollary of \cite[Theorem 14.7]{2012Algebra}.
\end{proof}

\begin{theorem}\label{HJplus}
  Suppose $A$ is a nonempty finite alphabet, $S$ is the free semigroup over $A$, $p$ is a minimal idempotent in $\beta S$ and $\langle C_n \rangle_{n=1}^\infty$ is a sequence in $p$. Then there exists a sequence $\langle w_n(v) \rangle_{n=1}^\infty$ of variable words over $A$ such that for each $t \in \mathbb{N}$, $Y_t \subseteq C_t$ and for each $F \in \mathcal{P}_f(\mathbb{N})$, $\sum_{t \in F}Y_t \subseteq C_{\min F}$, where $Y_t = \{w_t(a): a \in A \}$. 
\end{theorem}
\begin{proof}
  The proof is nearly the same as that of Theorem \ref{vdw} except that we need to apply Theorem \ref{HJ}.
\end{proof}

Also, one of the key that the proof of Theorem \ref{HJplus} can go through is that each $Y_t$ is finite, which is due to the condition ``$A$ is a nonempty finite alphabet''. So it is natural to ask whether Theorem \ref{HJplus} still holds if the alphabet is infinite. The answer is ``No'' like that to the question with respect to Theorem \ref{vdw}. Actually the Hales-Jewett theorem fails when the alphabet is infinite. It may be a known result, here we provide a proof.
\begin{theorem}
  If $A$ is an alphabet of size $\kappa \geq \omega$, $S$ is the free semigroup over $A$. Then there exists a finite coloring on $S$ such that no variable word $w(v)$ let $\{w(a): a \in A \}$ be monochromatic.
\end{theorem}
\begin{proof}
  For convenience, let $A = \kappa$. By Cantor's Normal Form Theorem (\cite[Theorem 2.26]{2003Set}), every nonzero ordinal $\alpha < \kappa$ can be represented uniquely in the form $\alpha = \sum_{i=1}^{n}\omega^{\beta_i} \cdot k_i$, where $n \geq 1$, $\alpha \geq \beta_1 > \ldots > \beta_n$ and $k_1, \ldots, k_n$ are nonzero natural numbers. For $\alpha < \kappa$, we say $\alpha$ is odd if $\alpha > 0$, $\beta_n = 0$ and $k_n$ is odd, otherwise we say it is even.
  
  Now define $c: S \rightarrow 2$ by setting $c(s) = 0$ if $\max(\mathrm{ran}(s))$ is even and $c(s) = 1$ otherwise for each $s \in S$. Then for any variable word $w(v)$, since it has finite length, $\mathrm{ran}(w(v)) \setminus \{ v \}$ is finite. Hence when $\kappa > \xi \geq \max(\mathrm{ran}(w(v)) \setminus \{ v \})$, we have $\max(\mathrm{ran}(w(\xi))) = \xi$. While $\xi$ may be even or not, it turns out that $\{w(a): a \in A \}$ is not monochromatic. 
\end{proof}

\section{Further Generalizations}
In this section, we will extend Phulara's result\cite[Theorem 2.6]{2015Ag} further. First we consider commutative semigroups.
\begin{theorem}\label{FG}
  Suppose $(S, +)$ is a commutative semigroup, $r \in J(S)$ is an idempotent and $R: \mathcal{P}_f({^\mathbb{N}}S) \rightarrow r$ is a function. Then there exist functions $\alpha: \mathcal{P}_f({^\mathbb{N}}S) \rightarrow S$ and $H: \mathcal{P}_f({^\mathbb{N}}S) \rightarrow \mathcal{P}_f(\mathbb{N})$ such that
  \begin{enumerate}
    \item if $F, G \in \mathcal{P}_f({^\mathbb{N}}S)$ and $F \subsetneq G$, then $\max H(F) < \min H(G)$, and
    \item if $m \in \mathbb{N}$, $G_1, \ldots, G_m \in \mathcal{P}_f({^\mathbb{N}}S)$, $G_1 \subsetneq \ldots \subsetneq G_m$ and $f_i \in G_i$ for each $i \in \{1, \ldots, m\}$, then $\sum_{i=1}^{m}(\alpha(G_i) + \sum_{t \in H(G_i)}f_i(t)) \in R(G_1)$.
  \end{enumerate}
\end{theorem}
\begin{proof}
  We define $\alpha$ and $H$ by induction on the size of $F \in \mathcal{P}_f({^\mathbb{N}}S)$ satisfying:
  \begin{enumerate}
    \item[(i)] if $\emptyset \neq G \subsetneq F$, then $\max H(G) < \max H(F)$, and
    \item[(ii)] if $m \in \mathbb{N}$, $G_1, \ldots, G_m \in \mathcal{P}_f({^\mathbb{N}}S)$, $\emptyset \neq G_1 \subsetneq \ldots \subsetneq G_m = F$ and $f_i \in G_i$ for each $i \in \{1, \ldots, m\}$, then $\sum_{i=1}^{m}(\alpha(G_i) + \sum_{t \in H(G_i)}f_i(t)) \in R(G_1)^\star$.
  \end{enumerate}
  First assume $F = \{ f \}$. Then since $R(F) \in r$, we have $R(F)^\star \in r$, hence $R(F)^\star$ is a J-set, then we pick some $a_0 \in S$ and $h_0 \in \mathcal{P}_f(\mathbb{N})$ such that
                               \[ a_0 + \sum_{t \in h_0}f(t) \in R(F)^\star. \]
  Then let $\alpha(F) = a_0$ and $H(F) = h_0$. So both hypothesis are satisfied.
  
  Now assume $|F| > 1$ and $\alpha(G)$ and $H(G)$ have been defined for all $G \subsetneq F$ so that the inductive hypothesis are satisfied. Let $K = \bigcup \{H(G): \emptyset \neq G \subsetneq F \}$, let $k = \max K$. Then for each nonempty $G \subsetneq F$, define $M_G = \{\sum_{i=1}^{m}(\alpha(G_i) + \sum_{t \in H(G_i)}f_i(t)): m \in \mathbb{N}, G = G_1 \subsetneq \ldots\subsetneq G_m \subsetneq F$ and for each $i \in \{ 1, \ldots, m \}, f_i \in G_i \}$. Note that each $M_G$ is finite and $M_G \subseteq R(G)^\star$ by hypothesis. Let \[
  A = R(F)^\star \cap \bigcap_{\emptyset \neq G \subsetneq F}\bigcap_{x \in M_G}(-x + R(G)^\star).   \]
  It is easy to see that $A \in r$. Then by \cite[Lemma 14.8.2]{2012Algebra}, we pick $a \in S$ and $h \in \mathcal{P}_f(\mathbb{N})$ with $\min h > k$ such that for each $f \in F$, $a + \sum_{t \in h}f(t) \in A$. Then let $\alpha(F) = a$ and let $H(F) = h$.
  
  Now let us verify that (i) and (ii) holds for $F$. Since $\min h > k$, (i) holds. Now take $m \in \mathbb{N}$, $G_1, \ldots, G_m \in \mathcal{P}_f({^\mathbb{N}}S)$, $\emptyset \neq G_1 \subsetneq \ldots \subsetneq G_m = F$ and $f_i \in G_i$ for each $i \in \{1, \ldots, m\}$. If $m = 1$, then $G_1 = F$, so $\sum_{i=1}^{m}(\alpha(G_i) + \sum_{t \in H(G_i)}f_i(t)) = \alpha(G_1) + \sum_{t \in H(G_1)}f_1(t) = a + \sum_{t \in h}f_1(t) \in A \in R(F)^\star =  R(G_1)^\star$. Otherwise, $m > 1$. Let $y = \sum_{i=1}^{m-1}(\alpha(G_i) + \sum_{t \in H(G_i)}f_i(t))$, so $y \in M_{G_1}$. Then $a + \sum_{t \in h}f_m(t) \in A \subseteq -y + R(G_1)^\star$, it turns out $\sum_{i=1}^{m}(\alpha(G_i) + \sum_{t \in H(G_i)}f_i(t)) = y + (a + \sum_{t \in h}f_m(t)) \in  R(G_1)^\star$, so (ii) holds.
\end{proof}

Actually this kind of generalization also works in noncommutative semigroups. And the proof is essentially the same as that in commutative semigroups, so we omit it. See \cite[Definition 3.1, Definition 3.3]{2015Ag} for the definition of $\mathcal{J}_m$ for $m \in \mathbb{N}$ and $x(m, a, t, f)$ for $m \in \mathbb{N}$, $a \in S^{m+1}$, $t \in \mathcal{J}_m$ and $f \in \mathcal{P}_f({^\mathbb{N}}S)$.
\begin{theorem}\label{noncom}
  Suppose $(S, \cdot)$ is a semigroup, $r$ is an idempotent in $J(S)$ and $R: \mathcal{P}_f({^\mathbb{N}}S) \rightarrow r$ is a function. Then there exist $m \in \mathcal{P}_f({^\mathbb{N}}S) \rightarrow \mathbb{N}$, $\alpha \in \times_{F \in \mathcal{P}_f({^\mathbb{N}}S)}S^{m(F) + 1}$ and $\tau \in \times_{F \in \mathcal{P}_f({^\mathbb{N}}S)}\mathcal{J}_{m(F)}$ such that 
  \begin{enumerate}
    \item if $F, G \in \mathcal{P}_f({^\mathbb{N}}S)$ and $F \subsetneq G$, then $\tau(F)(m(F)) < \tau(G)(1)$, and
    \item if $m \in \mathbb{N}$, $G_1, \ldots, G_m \in \mathcal{P}_f({^\mathbb{N}}S)$, $G_1 \subsetneq \ldots \subsetneq G_m$ and $f_i \in G_i$ for each $i \in \{1, \ldots, m\}$, then $\prod_{i=1}^{m}x(m(G_i), \alpha(G_i), \tau(G_i), f_i) \in R(G_1)$.
  \end{enumerate}
\end{theorem}
 
\section*{Acknowledgements}
   I acknowledge support by NSFC No. 12401002.

\bibliographystyle{plain}

\end{document}